\documentclass[12pt]{amsart}

\usepackage{amsmath, amscd, amssymb,xypic, euler}
\usepackage[frame,cmtip,arrow,matrix,line,graph,curve]{xy}
\usepackage{graphpap, color, pstricks}
\usepackage[mathscr]{eucal}
\usepackage
[
	pdftex,
	colorlinks,
	backref=page,
	citecolor=linkblue,
	linkcolor=linkblue
]{hyperref}
\definecolor{linkblue}{rgb}{0,0.2,0.6}

\newtheorem{theorem}{Theorem}[section]
\newtheorem{proposition}[theorem]{Proposition}
\newtheorem{corollary}[theorem]{Corollary}
\newtheorem{lemma}[theorem]{Lemma}

\theoremstyle{definition}
\newtheorem{definition}[theorem]{Definition}

\newtheorem{remark}[theorem]{Remark}

\newcommand{\LL}{\mathbb{L}}

\newcommand{\PP}{\mathbb{P}}

\newcommand{\cA}{\mathcal{A} }
\newcommand{\cB}{\mathcal{B} }

\newcommand{\cM}{\mathcal{M} }
\newcommand{\cO}{\mathcal{O} }

\newcommand{\cU}{\mathcal{U} }

\newcommand{\proj}{\mathrm{Proj}\;}

\def\cMzn{\overline{\cM}_{0,n} }
\def\cMgn{\overline{\cM}_{g,n} }
\def\cMg{\overline{\cM}_{g} }

\def\cMga{\overline{\cM}_{g, \cA} }
\def\cUga{\cU_{g,\cA} }

\def\Mg{\overline{M}_{g} }
\def\Mza{\overline{M}_{0,\cA} }
\def\Mga{\overline{M}_{g,\cA} }

\begin{document}

\title[Log canonical models for $\overline{\mathcal{M}}_{g,n}$]
{Log canonical models for $\overline{\mathcal{M}}_{g,n}$}
\date{November, 2011}
\author{Han-Bom Moon}
\address{Department of Mathematics, University of Georgia, Athens, 
GA 30602, USA}
\email{hbmoon@math.uga.edu}

\begin{abstract}
We prove a formula of log canonical models for 
moduli space $\overline{\mathcal{M}}_{g,n}$ of pointed stable curves
which describes all Hassett's 
moduli spaces of weighted pointed stable curves in a single equation.
This is a generalization of the preceding result for genus zero to all genera.
\end{abstract}

\maketitle


\section{Introduction}
\label{sec-introduction}

A central problem in algebraic geometry when studying a variety
$X$ is to determine all birational models of $X$.  One way to approach
this problem is to use divisors that have many sections. For
example, for a big divisor $D$ on $X$, one can hope to define
and learn about a natural birational model
\[
	X(D) := \proj \bigoplus_{k \ge 0}H^{0}(X, \cO(\lfloor kD \rfloor)).
\]
Many results in birational geometry in last several decades are about 
overcoming of technical difficulties such as the finite generation of 
the section ring.

The moduli spaces $\cMg$ and $\cMgn$ of stable 
curves and stable curves with marked points, 
are important as they give information about smooth curves and their
degenerations.  Moreover, as special varieties, they have played a
useful role in illustrating and testing the goals of birational
geometry for example the \emph{minimal model program}.  

In this paper we show the following theorem, 
as an example of log minimal model program applied to 
the moduli space $\cMgn$.
\begin{theorem}\label{thm-mainthmintro}
Let $\cA = (a_{1}, a_{2}, \cdots, a_{n})$ be a weight datum
satisfying $2g - 2 + \sum_{i=1}^{n}a_{i} > 0$. 
Then 
\[
	\cMgn(K_{\cMgn} + 11\lambda + \sum_{i=1}^{n}a_{i}\psi_{i}) 
	\cong \Mga,
\]
where $\Mga$ is the coarse moduli space of Hassett's moduli space $\cMga$ of weighted pointed stable curves with weight datum $\cA$ (\cite{Has03}).
\end{theorem}

This is proved for genus zero in \cite{Moo11a}.  In this article we
establish the result in all genera. 
A key step is to construct ample divisors on the moduli spaces
$\cMga$, for $g > 0$ (Proposition \ref{prop-ample}).

To put these results into context, we next give some history of this problem.
By \cite{HM82, Har84}, we know that
for $g \ge 24$, the canonical divisor
$K_{\cMg}$ is big, and by \cite{BCHM10}, that the
canonical model $\cMg(K_{\cMg})$ exists.
The hope is that one will be able to describe the canonical
model as a moduli space itself, and this problem has attracted a great
deal of attention.  

One approach has been the \emph{Hassett-Keel program}.
By \cite{CH88}, it is well known that $\cMg(K_{\cMg} + D) \cong \Mg$,
where $D = \cMg - \cM_{g}$ is the sum of all boundary divisors. 
So if we figure out the log canonical models
$\cMg (K_{\cMg} + \alpha D)$ for $0 \le \alpha \le 1$
and find the variation of log canonical models during 
we reduce the coefficient $\alpha$ from $1$ 
to $0$, we can finally obtain the canonical model. 
Still this problem is far from complete except small genera cases
(\cite{Has05, HL10}),
we have understood many different compactifications of $\cM_{g}$.
For example, see \cite{HH08, HH09, FS10}. 

We can perform a similar program for $\cMgn$, the moduli space of pointed 
stable curves.
The first result in this direction is the thesis of M. Simpson (\cite{Sim08}).
He studied the log canonical model of $\cMzn$ assuming the F-conjecture. 
He proved that in a suitable range of $\beta$,
$\cMzn(K_{\cMzn}+\beta D)$ is isomorphic to $\Mza$ where
$\cA$ is symmetric weight datum which depends on $\beta$.
This theorem was later proved without assuming the F-conjecture 
(\cite{AS08, FS11, KM11}). 
For $g = 1$, there are results of Smyth (\cite{Smy10, FS10}) considering 
birational models of type 
$\overline{\cM}_{1,n}(s\lambda + t \sum_{i=1}^{n} \psi_{i} - D)$. 
In this case, birational models are given by moduli spaces 
of symmetric weighted curves with even worse singularities. 
In \cite{Fed11a} Fedorchuk showed that for any genus $g$ and weight 
datum $\cA = (a_{1}, a_{2}, \cdots, a_{n})$ 
satisfying $2g - 2 + \sum_{i=1}^{n}a_{i} > 0$, 
there is a divisor $D_{g, \cA}$ on $\cMgn$ such that 
(1) $(\cMgn, D_{g, \cA})$ is a \emph{lc pair} and (2) 
$\cMgn(K_{\cMgn}+D_{g, \cA}) \cong \Mga$.

Both formula in Theorem \ref{thm-mainthmintro} and \cite{Fed11a} have 
their own interest.
Theorem \ref{thm-mainthmintro} says that 
the \emph{same weight datum} determines the log canonical model of 
parameterized curves and that of the parameter space itself. 
Indeed, for any weight datum $\cA$, 
there is a reduction morphism $\varphi_{\cA} : \cMgn \to 
\cMga$ (\cite[Theorem 4.1]{Has03}). 
For a stable curve $(C, x_{1}, \cdots, x_{n}) \in \cMgn$, its image 
$\varphi_{\cA}(C)$ is given by the \emph{log canonical model}
\begin{equation}\label{eqn-lcmodelfiber}
	C(\omega_{C} + \sum_{i=1}^{n} a_{i}x_{i}) := 
	\proj \bigoplus_{k \ge 0}H^{0}(C, \cO(\lfloor k(\omega_{C} + 
	\sum_{i=1}^{n} a_{i}x_{i})\rfloor))
\end{equation}
of $C$.

In Section \ref{sec-tautologicaldivisor} we list the definition and 
computational results of several tautological divisors. 
In Section \ref{sec-positivity}, we prove Proposition \ref{prop-positivity}
which is the crucial step of the proof of Theorem \ref{thm-mainthmintro}.
We give the rest of the proof in Section \ref{sec-proof}.

\textbf{Acknowledgements.} We thank to Angela Gibney 
for many invaluable suggestions for the first draft of this paper.

\section{A glossary of divisors on $\cMga$}
\label{sec-tautologicaldivisor}

In this section, we recall definitions of tautological divisors on $\cMga$ 
and their push-forward/pull-back formulas.
We omit proofs because they are results of simple generalizations of 
well-known results. For the proofs and ideas, see for instance 
\cite{AC96, AC98, HM98, Has03, Fed11a, FS11, Moo11a}.

\begin{definition}\label{dfn-tautologicaldivisor}
Fix a weight datum $\cA = (a_{1}, a_{2}, \cdots, a_{n})$. 
Let $[n] := \{1, 2, \cdots, n\}$. For $I \subset [n]$, let 
$w_{I}:= \sum_{i \in I}a_{i}$. 
Let $\pi : \cUga \to \cMga$ be the universal family and 
$\sigma_{i} : \cMga \to \cUga$ for $i = 1, 2, \cdots, n$ be the universal 
sections. Also let $\omega := \omega_{\cUga/\cMga}$ be the relative 
dualizing sheaf.

\begin{enumerate}
	\item The \emph{kappa class}: 
	$\kappa := \pi_{*}(c_{1}^{2}(\omega))$. 
	Our definition is different from several others for example  
	\cite{AC96, AC98, ACG11}.
	\item The \emph{Hodge class}:
	$\lambda := c_{1}(\pi_{*}(\omega))$.
	\item The \emph{psi classes}: For $i = 1, 2, \cdots, n$, 
	let $\LL_{i}$ be the line bundle on $\cMga$, whose fiber over 
	$(C, x_{1}, x_{2}, \cdots, x_{n})$ is $\Omega_{C}|_{x_{i}}$, 
	the cotangent space at $x_{i}$ in $C$.
	The \emph{$i$-th psi class} $\psi_{i}$ is $c_{1}(\LL_{i})$. 
	On the other hand, $\psi_{i}$ can be defined in terms of intersection theory.
	$\psi_{i} = \pi_{*}(\omega \cdot \sigma_{i}) = 
	-\pi_{*}(\sigma_{i}^{2})$.
	The \emph{total psi class} is $\psi := \sum_{i=1}^{n}\psi_{i}$.
	\item \emph{Boundaries of nodal curves}:
	Take a pair $(j, I)$ for $0 \le j \le g$ and $I \subset [n]$.
	Suppose that if $j = 0$, then $w_{I} > 1$.
	Let $D_{j, I} \subset \cMga$ be the closure of the locus of curves with 
	two irreducible components $C_{j, I}$ and $C_{g-j, I^{c}}$ such that 
	$C_{j, I}$ (resp. $C_{g-j, I^{c}}$ is a smooth genus $j$ 
	(resp. $g-j$) curve and $x_{i} \in C_{j, I}$ if and only if $i \in I$.
	For a notational convenience, set $D_{j, I} = 0$ when 
	$j = 0$ and $|I| \le 1$.
	Let $D_{irr}$ be the closure of the locus of irreducible nodal curves. 
	Let $D_{nod}$ be the sum of all $D_{j, I}$ and $D_{irr}$.
	\item \emph{Boundaries of curves with coincident sections}:
	Suppose that $I = \{i, j\}$ and $w_{I} \le 1$. 
	Let $D_{i=j}$ be the locus of curves with $x_{i} = x_{j}$. 
	$D_{i=j}$ is equal to $\pi_{*}(\sigma_{i} \cdot \sigma_{j})$.
	Let $D_{sec}$ be the sum of all boundaries of curves with 
	coincident sections.
\end{enumerate}
\end{definition}

The canonical divisor $K_{\cMga}$ is computed by Hassett 
(\cite[Section 3.3.1]{Has03}). By Mumford's relation 
$\kappa = 12\lambda - D_{nod}$ (\cite[Theorem 5.10]{Mum77}), 
it has two different presentations.

\begin{lemma}\label{lem-canonicaldivisor}
\cite[Section 3.3.1]{Has03}
\[
	K_{\cMga} = \frac{13}{12}\kappa - \frac{11}{12}D_{nod}
	+ \psi = 13\lambda - 2D_{nod} + \psi.
\]
\end{lemma}

Next, we present the push-forward and pull-back formulas we will often use. 

Let $\cA = (a_{1}, a_{2}, \cdots, a_{n})$ and $\mathcal{B} = 
(b_{1}, b_{2}, \cdots, b_{n})$ be weight data such that 
$a_{i} \ge b_{i}$ for all $i = 1, 2, \cdots, n$. 
For $I \subset [n]$, set $w_{I}^{\cA} = \sum_{i \in I}a_{i}$ and 
$w_{I}^{\mathcal{B}} = \sum_{i \in I}b_{i}$.

\begin{lemma}\label{lem-pushforwardpullbackreduction}
Let $\varphi_{\cA, \cB} : \cMga \to \overline{\cM}_{g, \cB}$ 
be the reduction morphism (\cite[Theorem 4.1]{Has03}).
\begin{enumerate}
	\item $\displaystyle\varphi_{\cA, \cB *}(\kappa) = 
	\kappa - \sum_{w_{\{j, k\}}^{\cA} > 1,\; w_{\{j, k\}}^{\cB}
	\le 1}D_{j=k}$.
	\item $\varphi_{\cA, \cB *}(\lambda) = \lambda$.
	\item $\displaystyle\varphi_{\cA, \cB *}(\psi_{i}) = \psi_{i}
	+ \sum_{w_{\{i, j\}}^{\cA} > 1,\; w_{\{i, j\}}^{\cB} \le 1} 
	D_{i=j}$.
	\item $\varphi_{\cA, \cB *}(D_{i, I}) = 
	\begin{cases}
	0, & i = 0, |I| \ge 3, w_{I}^{\cB} \le 1,\\
	D_{j=k}, & i = 0, I = \{j, k\}, w_{I}^{\cB} \le 1,\\
	D_{i, I}, & \mbox{otherwise}.
	\end{cases}$
	\item $\varphi_{\cA, \cB *}(D_{irr}) = D_{irr}$.
	\item $\varphi_{\cA, \cB *}(D_{j=k}) = D_{j=k}$.
	\item $\displaystyle \varphi_{\cA, \cB}^{*}(\kappa) = 
	\kappa + \sum_{w_{I}^{\cB} \le 1,\; w_{I}^{\cA} > 1}D_{0, I}$.
	\item $\varphi_{\cA, \cB}^{*}(\lambda) = \lambda$.
	\item $\displaystyle\varphi_{\cA, \cB}^{*}(\psi_{i}) = \psi_{i}
	- \sum_{i \in I,\; w_{I}^{\cB} \le 1,\; w_{I}^{\cA} > 1}D_{0, I}$.
	\item $\varphi_{\cA, \cB}^{*}(D_{i, I}) = D_{i, I}$. 
	\item $\varphi_{\cA, \cB}^{*}(D_{irr}) = D_{irr}$. 
	\item $\displaystyle \varphi_{\cA, \cB}^{*}(D_{j=k}) = 
	\begin{cases}\displaystyle
	D_{j=k} + \sum_{I \supset \{j, k\},\; w_{I}^{\cB} \le 1}D_{0, I},
	& w_{\{j, k\}}^{\cA} \le 1,\\
	\displaystyle \sum_{I \supset \{j, k\},\; w_{I}^{\cB} \le 1}D_{0, I},
	& w_{\{j, k\}}^{\cA} > 1.\\
	\end{cases}$
\end{enumerate}

\end{lemma}

The special case $\varphi_{(1, 1, \cdots, 1), \cA} : 
\overline{\cM}_{g, (1, 1, \cdots, 1)} \to \cMga$ is particularly important. 
For notational convenience, let $\varphi_{\cA} := 
\varphi_{(1, 1, \cdots, 1), \cA}$.

\begin{corollary}\label{cor-pushforwardpullbackreduction}
For  $\varphi_{\cA} : \cMgn = \overline{\cM}_{g, (1, 1, \cdots, 1)} 
\to \cMga$,
\begin{enumerate}
	\item $\varphi_{\cA *}(\kappa) = \kappa - D_{sec}$.
	\item $\varphi_{\cA *}(\lambda) = \lambda$.
	\item $\displaystyle \varphi_{\cA *}(\psi_{i}) = 
	\psi_{i} + \sum_{w_{\{i, j\}}\le 1}D_{i=j}$.
	\item $\varphi_{\cA *}(D_{i, I}) = 
	\begin{cases} 
	0, & i = 0, |I| \ge 3, w_{I} \le 1,\\
	D_{I}, & i = 0, |I| = 2, w_{I} \le 1,\\
	D_{i, I}, & \mbox{otherwise}.
	\end{cases}$
	\item $\varphi_{\cA *}(D_{irr}) = D_{irr}$.
	\item $\displaystyle \varphi_{\cA}^{*}(\kappa) = 
	\kappa + \sum_{w_{I} \le 1}D_{0, I}$.
	\item $\varphi_{\cA}^{*}(\lambda) = \lambda$.
	\item $\displaystyle \varphi_{\cA}^{*}(\psi_{i}) = \psi_{i} 
	- \sum_{i \in I,\; w_{I} \le 1}D_{0, I}$.
	\item $\varphi_{\cA}^{*}(D_{i, I}) = D_{i, I}$.
	\item $\varphi_{\cA}^{*}(D_{irr}) = D_{irr}$.
	\item $\displaystyle \varphi_{\cA}^{*}(D_{j=k}) = 
	\sum_{I \supset \{j, k\},\; w_{I} \le 1}D_{0, I}$.
\end{enumerate}
\end{corollary}

\begin{lemma}
Let $\rho : \overline{\cM}_{g, \cA \cup \{a_{p}\}} \to \cMga$ be the 
forgetful morphism (\cite[Theorem 4.3]{Has03}).
\begin{enumerate}
	\item $\displaystyle \rho^{*}(\kappa) = \kappa 
	+ \sum_{w_{\{i, p\}} > 1}D_{0, \{i, p\}}$.
	\item $\rho^{*}(\lambda) = \lambda$.
	\item $\rho^{*}(\psi_{i}) = 
	\begin{cases}
	\psi_{i}, & w_{\{i, p\}} \le 1,\\
	\psi_{i} - D_{0, \{i, p\}}, & w_{\{i, p\}} > 1.
	\end{cases}$
	\item $\rho^{*}(D_{i, I}) = D_{i, I} + D_{i, I \cup \{p\}}$.
	\item $\rho^{*}(D_{irr}) = D_{irr}$.
	\item $\rho^{*}(D_{j=k}) = 
	\begin{cases}
	D_{j=k}, & w_{\{j, k, p\}} \le 1,\\
	D_{j=k} + D_{0, \{j, k, p\}}, & w_{\{j, k, p\}} > 1.
	\end{cases}$
\end{enumerate}
\end{lemma}

For $I = \{j_{1}, j_{2}, \cdots, j_{r}\} \subset [n]$, 
let $D_{i, I}$ be a boundary of nodal curves. 
Set $I^{c} = \{k_{1}, k_{2}, \cdots, k_{s}\}$. Then 
$D_{i, I}$ is isomorphic to $\overline{\cM}_{i, \cA_{I}} \times 
\overline{\cM}_{g-i, \cA_{I^{c}}}$ where 
$\cA_{I} = (a_{j_{1}}, a_{j_{2}}, \cdots, a_{j_{r}}, 1)$
and $\cA_{I^{c}} = (a_{k_{1}}, a_{k_{2}}, \cdots, a_{k_{s}}, 1)$. 
Let $\eta_{i, I} : \overline{\cM}_{i, \cA_{I}} \times 
\overline{\cM}_{g-i, \cA_{I^{c}}} \cong D_{i, I} \hookrightarrow \cMga$ 
be the inclusion. Let $\pi_{\ell}$ for $\ell = 1, 2$ be the projection from 
$\overline{\cM}_{i, \cA_{I}} \times \overline{\cM}_{g-i, \cA_{I^{c}}}$
to the $\ell$-th component.

\begin{lemma}\label{lem-restrictiontonodalboundary}
Let $p$ (resp. $q$) be the last index of $\cA_{I}$ (resp. $\cA_{I^{c}}$) 
with weight one. 
\begin{enumerate}
	\item $\eta_{i, I}^{*}(\kappa) = \pi_{1}^{*}(\kappa + \psi_{p}) 
	+ \pi_{2}^{*}(\kappa + \psi_{q})$.
	\item $\eta_{i, I}^{*}(\lambda) = \pi_{1}^{*}(\lambda) 
	+ \pi_{2}^{*}(\lambda)$.
	\item $\eta_{i, I}^{*}(\psi_{j}) = 
	\begin{cases}
	\pi_{1}^{*}(\psi_{j}), & j \in I,\\
	\pi_{2}^{*}(\psi_{j}), & j \in I^{c}.
	\end{cases}$
	\item $\eta_{i, I}^{*}(D_{j, J}) = 
	\begin{cases}
	-\pi_{1}^{*}(\psi_{p}) - \pi_{2}^{*}(\psi_{q}), & D_{i, I} = D_{j, J},\\
	\pi_{1}^{*}(D_{j, J}), & j \le i, J \subset I, D_{i, I} \ne D_{j, J},\\
	\pi_{1}^{*}(D_{g-j, J^{c}}), & g-j \le i, J^{c} \subset I,
	D_{i, I} \ne D_{j, J},\\
	\pi_{2}^{*}(D_{j, J}), & j \le g-i, J \subset I^{c},
	D_{i, I} \ne D_{j, J},\\
	\pi_{2}^{*}(D_{g-j, J^{c}}), & i \le j, I \subset J, D_{i, I} \ne D_{j, J},\\
	0, & \mbox{otherwise}.
	\end{cases}$
	\item $\eta_{i, I}^{*}(D_{irr}) = \pi_{1}^{*}(D_{irr}) + 
	\pi_{2}^{*}(D_{irr})$.
	\item $\eta_{i, I}^{*}(D_{j=k}) = 
	\begin{cases}
	\pi_{1}^{*}(D_{j=k}), & j, k \in I,\\
	\pi_{2}^{*}(D_{j=k}), & j, k \notin I,\\
	0, & \mbox{otherwise}.
	\end{cases}$
\end{enumerate}
\end{lemma}

Let $(C, x_{1}, x_{2}, \cdots, x_{n}, p, q)$ be a genus $g-1$, 
$\cA \cup \{1, 1\}$-stable curve. 
By gluing $p$ and $q$, we obtain an $\cA$-stable curve of genus $g$. 
Since this gluing operation is extended to families of curves and functorial, 
we obtain a morphism 
$\xi : \overline{\cM}_{g-1, \cA \cup \{1, 1\}} \to \cMga$. 
Moreover, $\xi$ is an embedding and the image is precisely $D_{irr}$.

\begin{lemma}\label{lem-restrictiontoirreduciblenodalboundary}
Let $\xi$ be the gluing map and let $p$, $q$ be two identified sections. 
\begin{enumerate}
	\item $\xi^{*}(\kappa) = \kappa + \psi_{p} + \psi_{q}$.
	\item $\xi^{*}(\lambda) = \lambda$.
	\item $\xi^{*}(\psi_{i}) = \psi_{i}$.
	\item $\xi^{*}(D_{i, I}) = D_{i, I} + D_{i-1, I \cup \{p, q\}}$.
	\item $\displaystyle \xi^{*}(D_{irr}) = D_{irr} - \psi_{p} - \psi_{q} + 
	\sum_{p \in I,\; q \notin I}D_{i, I}$.
	\item $\xi^{*}(D_{j=k}) = D_{j=k}$.
\end{enumerate}
\end{lemma}

Finally, for a nonempty subset $I \subset [n]$, assume that $w_{I} \le 1$. 
Let $\cA'$ be a new weight datum defined by replacing $(a_{i})_{i \in I}$
with a single rational number $w_{I} = \sum_{i \in I}a_{i}$. 
We can define an embedding $\chi_{I} : \overline{M}_{g, \cA'} \to 
\cMga$ which sends an $\cA'$-stable curve to the $\cA$-stable curve 
obtained by replacing the point of weight $w_{I}$ 
with $|I|$ points of weight $(a_{i})_{i \in I}$ on the same position. 

\begin{lemma}\label{lem-restrictiontoboundaryofsection}
Let $\chi_{I} : \overline{\cM}_{g, \cA'} \to \cMga$ be the replacing morphism.
Let $p$ be the unique index of $\cA'$ replacing indices in $I$.
\begin{enumerate}
	\item $\chi_{I}^{*}(\kappa) = \kappa$.
	\item $\chi_{I}^{*}(\lambda) = \lambda$.
	\item $\chi_{I}^{*}(\psi_{i}) = \begin{cases}
	\psi_{i}, & i \notin I,\\
	\psi_{p}, & i \in I.
	\end{cases}$
	\item $\chi_{I}^{*}(D_{nod}) = D_{nod}$.
	\item $\chi_{I}^{*}(D_{irr}) = D_{irr}$.
	\item $\chi_{I}^{*}(D_{j=k}) = \begin{cases}
	D_{j=k}, & j, k \notin I,\\
	D_{j=p}, & j \notin I, k \in I,\\
	-\psi_{p}, & j, k \in I.
	\end{cases}$
\end{enumerate}
\end{lemma}

\section{A positivity result on families of curves}
\label{sec-positivity}

A key step of the proof of Theorem \ref{thm-mainthmintro} is to construct
an ample divisor on $\cMga$. 
In this section, we prove the following technical positivity result of a divisor,
which will be used to the proof of the main theorem. 

\begin{proposition}\label{prop-positivity}
Fix a weight datum $\cA = (a_{1}, a_{2}, \cdots, a_{n})$ and 
a positive genus $g$. 
Let $B$ be an integral curve. Let $\pi : \cU \to B$ be a flat family of 
$\cA$-stable genus $g$ curves and let $\sigma_{i} : B \to \cU$ 
for $i = 1, 2, \cdots, n$ be $n$ sections.
Suppose that a general fiber of $\pi$ is smooth. 
Then there exists a positive rational 
number $\epsilon_{g, \cA} > 0$ which depends only on $g$ and $\cA$ 
such that 
\begin{equation}\label{eqn-kappapsiinequality}
	(2\kappa + \psi) \cdot B \ge \epsilon_{g, \cA}\cdot \mathrm{mult}_{x}B
\end{equation}
for any point $x \in B$.
\end{proposition}

\begin{remark}
\begin{enumerate}
	\item Proposition \ref{prop-positivity} does \emph{not} imply 
	that $2\kappa + \psi$ 
	is ample on $\cMga$ even though the statement is similar to Seshadri's 
	ampleness criterion (Theorem \ref{thm-Seshadricriterion}). 
	Note that there is an assumption that a general fiber should be smooth.
	\item Proposition \ref{prop-positivity} is not true when $g = 0$. 
	Indeed, $K_{\cMzn} \equiv 2\kappa + \psi$
	(\cite[Lemma 2.6]{Moo11a}).
	It is well-known that for $n = 4, 5$, 
	$K_{\cMzn}$ is anti-ample. So it intersects \emph{negatively} 
	with every curve.
\end{enumerate}
\end{remark}

We will use following two positivity results. 

\begin{theorem}[Seshadri's criterion, {\cite[Theorem 1.4.13]{Laz04a}}]
\label{thm-Seshadricriterion}
Let $X$ be a projective variety and $D$ is a divisor on $X$.
Then $D$ is ample if and only if there exists a positive number $\epsilon > 0$
such that 
\[
	D \cdot C \ge \epsilon \cdot \mathrm{mult}_{x}C
\]
for every point $x \in C$ and every irreducible curve $C \subset X$.
\end{theorem}

\begin{theorem}\label{thm-Cornalba}
\cite[Lemma 3.2]{Cor93}
There are positive integers $h$ and $M$ depending on $r$ and $d$,
such that the following statement 
holds for any flat family $\pi : \cU \to B$ of nodal curves
over any integral curve $B$. Let $L$ be a relative 
degree $d$ line bundle on $\cU$. Suppose that $\pi$ is not isotrivial as a 
family of \emph{polarized} curves. Moreover, assume that 
\begin{enumerate}
	\item a general fiber is smooth,
	\item $R^{1}\pi_{*}(L^{i}) = 0$ for $i >> 0$ 
	and $r := \dim H^{0}(\cU_{b}, L_{\cU_{b}})$ 
	is independent of $b \in B$,
	\item For a general $b \in B$, $L_{\cU_{b}}$ is base-point-free, 
	very ample and embeds $\cU_{b}$ in $\PP^{r-1}$ as a Hilbert stable 
	subscheme. 
\end{enumerate}
Then 
\begin{equation}\label{eqn-multiplicityinequality}
\begin{split}
	&\left(\frac{r}{2}(L \cdot L) - d (\deg \pi_{*}(L))\right) h^{2}\\
	&+ \left((g-1)(\deg \pi_{*}(L)) - \frac{r}{2}(L \cdot \omega )\right) h
	+ r \deg \lambda \ge \frac{1}{M} \mathrm{mult}_{x}B
\end{split}
\end{equation}
for every point $x \in B$.
\end{theorem}

\begin{remark}
\begin{enumerate}
	\item In \cite{Cor93}, Cornalba assumed that $g \ge 2$, 
	but the proof of the theorem shows that this result is true 
	without the assumption. See \cite[Section 3]{Cor93} and 
	\cite[Section 2]{CH88}.
	\item If $d \ge 2g > 0$, then by \cite[Theorem 4.34]{HM98}
	a smooth curve is automatically Hilbert stable. 
	Note that the same proof can be applied to $g = 1$ case, too.
	\item Even if a given family of curves is isotrivial as a family of 
	\emph{abstract} curves, we can apply the theorem 
	if the family is not isotrivial as a family of \emph{polarized} curves. 
	\item The equation \eqref{eqn-multiplicityinequality} is different from 
	\cite[(3.3)]{Cor93}. 
	But if we follow the proof, the right formula is 
	\eqref{eqn-multiplicityinequality}.
\end{enumerate}
\end{remark}

\begin{proof}[Proof of Proposition \ref{prop-positivity}]
We will divide the proof into several steps. 

\medskip

\textsf{Step 1.} It is sufficient to show the result for a weight datum 
$n\cdot \tau = (\tau, \tau, \cdots, \tau)$ for sufficiently small $\tau > 0$
satisfying $n \cdot \tau \le 1$.

\medskip

Let $\varphi_{\cA, n \cdot \tau} : \cMga \to \overline{\cM}_{g, n \cdot \tau}$
be the reduction morphism.
By the assumption that a general fiber of $\pi$ is smooth, 
$\overline{B} := \varphi_{\cA, n \cdot \tau}(B)$ is an integral curve in 
$\overline{\cM}_{g, n \cdot \tau}$. 
By the projection formula, 
\begin{equation}
\begin{split}
	(2\kappa + \psi) \cdot \overline{B}
	= &\; \varphi_{\cA, n \cdot \tau}^{*}(2\kappa + \psi) \cdot B
	= (2\kappa + 2\sum_{w_{I} > 1}D_{0, I}
	+ \psi - |I| \sum_{w_{I} > 1}D_{0, I}) \cdot B\\
	= &\; (2\kappa + \psi) \cdot B 
	- (|I|- 2)\sum_{w_{I} > 1}D_{0, I} \cdot B
	\le (2\kappa + \psi) \cdot B,
\end{split}
\end{equation}
because $|I| \ge 2$.
Thus if the result is true for the weight datum $n \cdot \tau$, then 
\[
	(2\kappa + \psi) \cdot B \ge (2\kappa + \psi) \cdot \overline{B}
	\ge \epsilon_{g, n \cdot \tau}\cdot 
	\mathrm{mult}_{\varphi_{\cA, n \cdot \tau}(x)}
	\overline{B} \ge \epsilon_{g, n \cdot \tau}\cdot \mathrm{mult}_{x}B.
\]
Therefore if we define $\epsilon_{g, \cA} := \epsilon_{g, n \cdot \tau}$, 
the proposition holds.

\medskip

\textsf{Step 2.} We can reduce the number of sections.

\medskip

Let $\rho : \overline{\cM}_{g, n \cdot \tau} \to 
\overline{\cM}_{g, (n-1) \cdot \tau}$ be the forgetful morphism.
There are two possible cases. 
If $\overline{B} := \rho(B)$ is a curve, then 
\[
	(2\kappa + \psi) \cdot \overline{B} = 
	\rho^{*}(2\kappa + \psi) \cdot B = 
	(2\kappa + \psi) \cdot B.
\]
Thus $(2\kappa + \psi) \cdot B = (2\kappa + \psi) \cdot \overline{B} 
\ge \epsilon_{g, (n-1)\cdot \tau}\cdot \mathrm{mult}_{\pi(x)}\overline{B} \ge
\epsilon_{g, (n-1)\cdot \tau}\cdot \mathrm{mult}_{x}B$.

If $\rho(B)$ is a point, then the family $\pi : \cU \to B$ is isotrivial 
as a family of abstract curves after forgetting the last section.
For $g \ge 2$, we will use Theorem \ref{thm-Cornalba} 
with $L = \omega^{k}(\sigma_{n})$ 
for sufficiently large $k$. 
Then $d = 2k(g-1)+1$, $r = (2k-1)(g-1) +1$ and $L$ satisfies all 
assumptions in Theorem \ref{thm-Cornalba}.
Note that $(\pi : \cU \to B, L)$ is not isotrivial as a family of 
polarized curves.

By the Riemann-Roch theorem, 
\[
	\deg \pi_{*}(L) = \frac{(L \cdot L)}{2} - \frac{(L \cdot \omega)}{2}
	+ \deg \lambda,
\]
since $R^{1}\pi_{*}(L) = 0$.
Note that in our situation, $\deg \lambda = \lambda \cdot B = 0$ and 
$D_{nod} \cdot B = 0$, $\psi_{i} \cdot B = 0$ for $i = 1, 2, \cdots, n-1$.
Thus $\kappa \cdot B = 0$ by Mumford's relation and $\psi \cdot B 
= \psi_{n} \cdot B$.
So it is straightforward to check that 
\[
	\left(\frac{r}{2}(L \cdot L) - d (\deg \pi_{*}(L))\right) h^{2}
	+ \left((g-1)(\deg \pi_{*}(L)) - \frac{r}{2}(L \cdot \omega )\right) h
	+ r \deg \lambda
\]
\[
	 = \frac{gh(h-1)+2h}{2}\psi_{n}.
\]
Therefore 
\[
	(2\kappa +\psi) \cdot B = \psi_{n} \cdot B \ge \alpha\cdot 
	\mathrm{mult}_{x}B
\]
for some $\alpha > 0$ by Theorem \ref{thm-Cornalba}.

When $g = 1$, let $L = \cO(k\sigma_{n})$ for sufficiently large $k$.
Then $d = r = k$ and $L$ satisfies all assumptions in Theorem 
\ref{thm-Cornalba}.
In this case, 
\[
	\left(\frac{r}{2}(L \cdot L) - d (\deg \pi_{*}(L))\right) h^{2}
	+ \left((g-1)(\deg \pi_{*}(L)) - \frac{r}{2}(L \cdot \omega )\right) h
	+ r \deg \lambda
\]
\[
	= \frac{k^{2}h(h-1)}{2}\psi_{n}.
\]
So we can find $\epsilon_{g, n \cdot \tau} > 0$ by taking the minimum of 
these cases.

\medskip

\textsf{Step 3.} For $\overline{\cM}_{g, (\tau)} \cong \overline{\cM}_{g, 1}$, 
the proposition holds. 

\medskip

First of all, suppose that $g \ge 2$. 
Let $\rho : \overline{\cM}_{g, 1} \to \overline{\cM}_{g}$ be the forgetting 
morphism. If $\overline{B} = \rho(B)$ is a curve, then 
\[
	(2\kappa + \psi) \cdot B = 
	2\rho^{*}(\kappa) \cdot B + \psi \cdot B = 
	2\kappa \cdot \overline{B} + \pi_{*}(\omega \cdot \sigma_{1}).
\]
The divisor $\kappa$ is ample on $\cMg$ by \cite[Theorem 1.3]{CH88}. 
By Seshadri's criterion (Theorem \ref{thm-Seshadricriterion}), 
there is a positive number $\alpha > 0$
such that $\kappa \cdot \overline{B} \ge \alpha 
\cdot \mathrm{mult}_{x}\overline{B}$ 
for every irreducible curve $\overline{B}$ and $x \in \overline{B}$. 

On the other hand, let $\pi' : \cU' \to \overline{B}$ be the 
corresponding family of stable curves. 
Then there is a stabilization morphism $\tilde{\rho} : \cU \to \cU'$ 
and $\omega = \tilde{\rho}^{*}(\omega) + E$ where $E$ is an 
exceptional curve. 
Now 
\[
	\pi_{*}(\omega \cdot \sigma_{1}) = 
	\pi^{*}((\tilde{\rho}^{*}(\omega) + E) \cdot \sigma_{1})> 0
\]
because $\omega$ is ample on $\cU'$ by \cite[Proposition 3.2]{Ara71}, 
and $E \cdot \sigma_{1} > 0$. 

If $\pi : \cU \to B$ is isotrivial after forgetting the section $\sigma_{1}$, 
then by exactly same argument in \text{Step 2}, we can get the inequality
\eqref{eqn-kappapsiinequality}. 

On $\overline{\cM}_{1,1}$, $\kappa = 0$ and $\psi_{1} = 
\frac{1}{12}\lambda$ is ample (\cite[Theorem 2.2]{AC98}, 
note that $\kappa_{1}$ in \cite{AC98} is $\kappa + \psi_{1}$.)
Therefore we obtain $\epsilon_{1,(1)} > 0$ and the inequality 
\eqref{eqn-kappapsiinequality} by Seshadri's criterion.
\end{proof}

\section{Proof of the main theorem}
\label{sec-proof}

In this section, we prove our main result.

\begin{theorem}\label{thm-mainthm}
Let $\cA = (a_{1}, a_{2}, \cdots, a_{n})$ be a weight datum
satisfying $2g - 2 + \sum_{i=1}^{n}a_{i} > 0$. 
Then 
\[	
	\cMgn(K_{\cMgn} + 11\lambda + \sum_{i=1}^{n}a_{i}\psi_{i}) 
	\cong \Mga
\]
where $\Mga$ is the coarse moduli space of the moduli space $\cMga$ of 
$\cA$-stable curves.
\end{theorem}

\begin{remark}
\begin{enumerate}
	\item Theorem \ref{thm-mainthm} is a generalization of 
	\cite[Theorem 1.4]{Moo11a} because when $g = 0$, 
	the Hodge class $\lambda$ is trivial.
	\item Theorem \ref{thm-mainthm} suggests that 
	there is an unexpected relation between log canonical model of 
	moduli spaces and that of parameterized curves. 
	Giving a theoretical reason of this phenomenon would be interesting.
\end{enumerate}
\end{remark}

A key step of the proof is to construct an ample divisor on $\cMga$.
\begin{proposition}\label{prop-ample}
Let 
\[
	\Delta_{\cA} := K_{\cMga} + 11 \lambda + \sum_{i=1}^{n}a_{i}\psi_{i}
	= 2\kappa + \sum_{i=1}^{n}(1+a_{i})\psi_{i}.
\]
Then the push-forward $\varphi_{\cA *}(\Delta_{\cA})$ is ample.
\end{proposition}

\begin{proof}
By using definitions of tautological divisors and several formulas 
in Section \ref{sec-tautologicaldivisor}, 
it is straightforward to see that
\[
	\varphi_{\cA *}(\Delta_{\cA}) = 2\kappa 
	+ \sum_{i=1}^{n}(1+a_{i})\psi_{i} + \sum_{w_{\{i,j\}}\le 1} 
	w_{\{i,j\}}D_{i=j}
	= \pi_{*}((\omega + \sum_{i=1}^{n}a_{i}\sigma_{i})
	(2\omega + \sum_{i=1}^{n}\sigma_{i})).
\]

A key feature of $\varphi_{\cA *}(\Delta_{\cA})$ is that if we restrict it 
to boundaries, the result is also described the same formula. More precisely, 
by Lemma \ref{lem-restrictiontonodalboundary}, it is straightforward to check
\begin{equation}\label{eqn-restrictiontonodal}
	\eta_{i, I}^{*}(\varphi_{\cA *}(\Delta_{\cA}))
	= \pi_{1}^{*}(\varphi_{\cA_{I} *}(\Delta_{\cA_{I}}))
	+ \pi_{2}^{*}(\varphi_{\cA_{I^{c}} *}(\Delta_{\cA_{I^{c}}})).
\end{equation}
Also by Lemma \ref{lem-restrictiontoirreduciblenodalboundary}, 
\begin{equation}\label{eqn-restrictiontoirreduciblenodal}
	\xi^{*}(\varphi_{\cA *}(\Delta_{\cA})) = \varphi_{\cA\cup\{1,1\} *}
	(\Delta_{\cA \cup \{1,1\}}).
\end{equation}
Finally, for $I \subset [n]$ such that $w_{I} \le 1$, 
\begin{equation}\label{eqn-restrictiontobdrysection}
	\chi_{I}^{*}(\varphi_{\cA *}(\Delta_{\cA}))
	= \varphi_{\cA' *}(\Delta_{\cA'}) + (|I| - 1)
	\pi_{*}((\omega + \sum_{i \in J} a_{i}\sigma_{i})\cdot \sigma_{p})
\end{equation}
where $J$ is the index set for new weight datum $\cA'$
(See the notation for Lemma \ref{lem-restrictiontoboundaryofsection}.).

We will use Seshadri's criterion (Theorem \ref{thm-Seshadricriterion}) 
to show the ampleness of $\varphi_{\cA *}(\Delta_{\cA})$. 
For $\overline{\cM}_{1,1}$, it is straightforward 
to check the ampleness, and $g = 0$ is shown in \cite{Moo11a}.
So we can use the induction on the dimension of $\cMga$.

If $B$ is contained in a boundary of nodal curves, 
then $\varphi_{\cA *}(\Delta_{\cA}) \cdot B \ge \epsilon \cdot
\mathrm{mult}_{x}B$ 
by \eqref{eqn-restrictiontonodal} and \eqref{eqn-restrictiontoirreduciblenodal}.
If $B$ is in a boundary of coincident sections, 
then by the induction hypothesis, $\varphi_{\cA' *}(\Delta_{\cA'})$ is 
ample and $\pi_{*}((\omega + \sum_{i \in J}a_{i}\sigma_{i})
\cdot \sigma_{p})$ is nef by the third equation of \cite[Theorem 3.1]{Fed11a}.
Thus $\chi_{I}^{*}(\varphi_{\cA *}(\Delta_{\cA}))$ is ample by
\eqref{eqn-restrictiontobdrysection} and we can find $\epsilon > 0$.

So it is sufficient to check for the case that $B \cap \cM_{g, \cA} \neq 
\emptyset$. Let $\pi : \cU \to B$ with $\sigma_{i} : B \to \cU$ for 
$i = 1, 2, \cdots, n$ be the family of $\cA$-stable curves. 
We rewrite $\varphi_{\cA *}(\Delta_{\cA})$ as 
\[
	\varphi_{\cA *}(\Delta_{\cA}) = 
	\varphi_{\cA *}((\omega + \sum_{a_{i} = 1}\sigma_{i}
	+ \sum_{a_{i} < 1}a_{i}\sigma_{i})
	(2\omega + \sum_{i =1}^{n}\sigma_{i}))
\]
\[
	= \varphi_{\cA *}(((1-\delta)\omega + \sum_{a_{i} = 1}
	(1-\delta)\sigma_{i} + \sum_{a_{i} < 1}a_{i}\sigma_{i})
	(2\omega + \sum_{i=1}^{n}\sigma))
\]
\[
	+ \delta(\sum_{a_{i}=1}\sigma_{i})
	(2\omega + \sum_{i=1}^{n}\sigma_{i})
	+ \delta \omega (2\omega + \sum_{i=1}^{n}\sigma_{i})
\]
\[
	= \varphi_{\cA *}(((1-\delta)\omega + \sum_{a_{i} = 1}
	(1-\delta)\sigma_{i} + \sum_{a_{i} < 1}a_{i}\sigma_{i})
	(2\omega + \sum_{i=1}^{n}\sigma_{i}))
	 + \delta \sum_{a_{i} = 1}\psi_{i} + 
	\delta (2\kappa + \psi).
\]

Note that for there is $\delta > 0$ which depends on $g$ and $\cA$ such that 
\[
	\omega + \sum_{a_{i} = 1}\sigma_{i} + \sum_{a_{i}< 1}
	\frac{1}{1-\delta}a_{i}\sigma_{i}
\]
satisfies the assumption of \cite[Proposition 2.1]{Fed11a}. 
So it is nef on $\cU$. By \cite[Lemma 3.4]{Moo11a}, 
$2\omega + \sum_{i=1}^{n}\sigma_{i}$ is effective. 
Thus 
\[
	\varphi_{\cA *}(((1-\delta)\omega + \sum_{a_{i} = 1}
	(1-\delta)\sigma_{i} + \sum_{a_{i} < 1}a_{i}\sigma_{i})
	(2\omega + \sum_{i=1}^{n}\sigma))
\]
is nef on $B$. For the forgetful map $\rho : \cMga \to \cMg$, 
let $\pi' : \cU' \to \rho(B)$ be the corresponding family and 
$\sigma_{i}'$ be the image of section $\sigma_{i}$ on $\cU'$.
Then $\psi_{i} = -\sigma_{i}^{2} \ge - \sigma_{i}^{\prime 2} = \omega \cdot 
\sigma_{i}' \ge 0$ by \cite[Proposition 3.2]{Ara71}.
Finally, by Proposition \ref{prop-positivity}, there exists $\epsilon > 0$ 
depends only on $g$ and $\cA$ such that 
\[
	\varphi_{\cA *}(\Delta_{\cA}) \cdot B \ge \epsilon \cdot 
	\mathrm{mult}_{x}B
\]
for all $x \in B$.

There are only \emph{finitely} many boundary strata on $\cMga$. 
Therefore we can find the minimum of $\epsilon$ for all strata of $\cMga$
and we obtain an $\epsilon > 0$ for all curves in $\cMga$.
\end{proof}

Now Theorem \ref{thm-mainthm} is an immediate consequence of 
Proposition \ref{prop-ample}.

\begin{proof}[Proof of Theorem \ref{thm-mainthm}]
By Corollary \ref{cor-pushforwardpullbackreduction}, it is straightforward to 
check that 
\[
	\Delta_{\cA} = \varphi_{\cA}^{*}\varphi_{\cA *}(\Delta_{\cA})
	+ \sum_{w_{I} \le 1}(|I|-2)(1-w_{I})D_{0, I}.
\]
Note that $D_{0, I}$ with $|I| \ge 3$ and $w_{I} \le 1$ is an exceptional divisor
for $\varphi_{\cA}$.
Therefore $\Delta_{\cA}$ is a sum of the pull-back of an ample divisor and a
$\varphi_{\cA}$-exceptional effective divisor. 
Hence we obtain 
\[
	\cMgn(\Delta_{\cA}) 
	= \cMgn(\varphi_{\cA}^{*}\varphi_{\cA *}(\Delta_{\cA}))
	= \cMga(\varphi_{\cA *}(\Delta_{\cA})) = \Mga.
\]
See \cite[Proof of Theorem 3.1]{Moo11a} for more detail.
\end{proof}


\bibliographystyle{alpha}
\bibliography{Library}

\end{document}